\documentclass{aims}
\usepackage{amsmath}
\usepackage{amssymb}
\usepackage{mathtools}
\usepackage{paralist}
\usepackage{comment}
\usepackage{subcaption}
\usepackage{algorithm}
\usepackage{algpseudocode}
\usepackage{graphics} 
\usepackage{epsfig} 
\usepackage{graphicx}  
\usepackage[colorlinks=true]{hyperref}
\hypersetup{urlcolor=blue, citecolor=red}

\def\currentvolume{X}
 \def\currentissue{X}
  \def\currentyear{200X}
   \def\currentmonth{XX}
    \def\ppages{X--XX}
     \def\DOI{10.3934/xx.xx.xx.xx}

\def\bs{\boldsymbol}
\newcommand{\rom}[1]{\uppercase\expandafter{\romannumeral #1\relax}}
\DeclarePairedDelimiter\ceil{\lceil}{\rceil}

\newtheorem{theorem}{Theorem}[section]

\theoremstyle{definition}

\title[Detecting phase transitions in collective behavior] 
      {Detecting phase transitions in collective behavior using manifold's curvature}

\author[Kelum Gajamannage and Erik M. Bollt]{}


 \email{dineshhk@clarkson.edu}
 \email{ebollt@clarkson.edu}
 
\thanks{The authors were supported by the NSF grant CMMI-1129859. Erik M. Bollt was supported by the Army Research Office grant W911NF-12-1-276 and Office of Naval Research grant N00014-15-2093.}

\thanks{$^*$ Corresponding author.}

\begin{document}
\maketitle

\centerline{\scshape Kelum Gajamannage$^*$}
\medskip
{\footnotesize
 \centerline{Department of Mathematics}
   \centerline{Clarkson University}
   \centerline{Potsdam, NY-13699, USA}
} 

\medskip

\centerline{\scshape Erik M. Bollt}
\medskip
{\footnotesize
 \centerline{Department of Mathematics}
   \centerline{Clarkson University}
   \centerline{Potsdam, NY-13699, USA}
}

\bigskip

 \centerline{(Communicated by the associate editor name)}

\begin{abstract}
If a given behavior of a multi-agent system restricts the phase variable to a invariant manifold, then we define a phase transition as change of physical characteristics such as speed, coordination, and structure. We define such a phase transition as splitting an underlying manifold into two sub-manifolds with distinct dimensionalities around the singularity where the phase transition physically exists. Here, we propose a method of detecting phase transitions and splitting  the manifold into phase transitions free sub-manifolds. Therein, we utilize a relationship between curvature and singular value ratio of points sampled in a curve, and then extend the assertion into higher-dimensions using the shape operator. Then we attest that the same phase transition can also be approximated  by singular value ratios computed locally over the data in a neighborhood on the manifold. We validate the phase transitions detection method using one particle simulation and three real world examples.
\end{abstract}

\section{Introduction}\label{sec:introduction}
Multi-agent systems such as crowds of people \cite{helbing1997modelling, musse1997model, zhao2004tracking}, schools of fish \cite{gerlai2010high, partridge1982structure}, flocks of birds \cite{ballerini2008empirical, nagy2010hierarchical}, colonies of molds \cite{rappel1999self} and ants \cite{rauch1995pattern} often exhibit discrete phase transitions  due to variations of interaction among members \cite{beekman2001phase}. Specially, detecting phase transitions in crowds of people \cite{almeida2013change} is a popular research problem \cite{andrade2006hidden}. Abrupt changes upon variation of some parameters such as speed, coordination, and structure \cite{becco2006experimental, millonas1992swarms, vicsek1995novel} shift the phase of the system from one state to another \cite{deutsch1999principles, sole1996phase}. Numerous types of swarm decisions which determine the group dynamics are not only influenced by the intrinsic  social interaction among members \cite{couzin2005effective}, but also by some outside factors such as threats \cite{sumpter2008information} and presence of predators or food sources \cite{toffin2009shape}.  However, with a majority of data is recorded as videos, the classical approach of detecting phase transitions by means of tracking individuals and monitoring their dynamics is constrained by the size of the group and the scale of the problem \cite{mehran2009abnormal}. Being inspired by manifold representation of collective behavior, here we develop a method of detecting phase transitions in a multi-agent system.

In manifold theory, we can reveal an invariant manifold $\mathcal M$, in an abstract higher-dimensional space which describes the collective behavior of a group such that each frame partitioned from the collective behavior video corresponds to a point $\bs{p}$ on the manifold \cite{abaid2012topological}. In this setup, the whole group evolves according to the underlying flow
\begin{equation}
\Phi:{\mathcal M}\rightarrow {\mathcal M} \ ; \ \ \Phi(\bs{p}^{(t_1)})=\bs{p}^{(t_2)},
\end{equation}
where $\bs{p}^{(t_1)}$ and $\bs{p}^{(t_2)}$ are two points representing two consecutive frames at time-steps $t_1$ and $t_2$, respectively \cite{gajamannage2015dimensionality}.  Due to an abrupt behavioral change of the group, this mapping switches from one phase space to another and indicates a phase transition of the motion. Thus, in the presence of a phase transition, the system can be represented as two distinct sub-manifolds, $\mathcal{M}^{(j)}$ for $j=1$ and $2$, along with singularities where the phase transition physically exists \cite{gajamannageidentifying}. Herein, the most salient scenario is that the sub-manifolds intersect and make a locus $\mathcal{L}$ of singularities as
\begin{equation}\label{eqn:locus}
\mathcal{L}=\cap \mathcal{M}^{(j)}.
\end{equation}

As an example, we estimate dimensionalities of two distinct phases, walking and running, of a crowd of people given as a video in \cite{detection}. Two phases of the motion are embedded into two distinct manifolds as shown in Figure ~\ref{fig:crowd_walk_run}(a). We utilize an established dimensionality reduction routine called Isomap \cite{tenenbaum2000global} to obtain corresponding scaled residual variances of two embedding spaces (Fig. ~\ref{fig:crowd_walk_run}(b)) of each phase. Figure ~\ref{fig:crowd_walk_run}(b) shows that the two phases are embed on manifolds with different dimensionalities. This example acts as a proof-of-concept for developing an routine to detect phase transitions. Detecting phase transitions should be naively implemented on videos before utilize dimensionality reduction schemes such as Isomap \cite{tenenbaum2000global}, Local Linear Embedding \cite{roweis2000nonlinear}, as those may otherwise contain phase transitions.

\begin{figure}[hpt]
        	\centering
	\begin{subfigure}[b]{0.5\textwidth}
        	\includegraphics[width=1\textwidth]{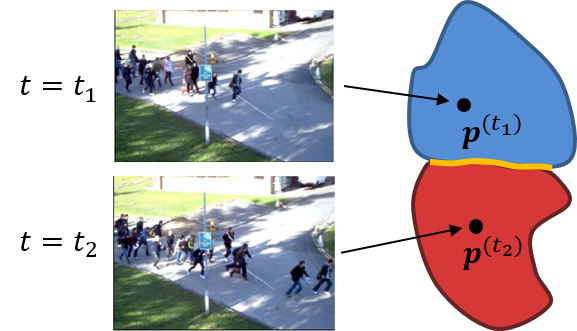}
        	\caption*{(a)}        	
    	\end{subfigure}   
	\quad
    	\begin{subfigure}[b]{0.4\textwidth}
	\vspace{20pt}
        	\includegraphics[width=1\textwidth]{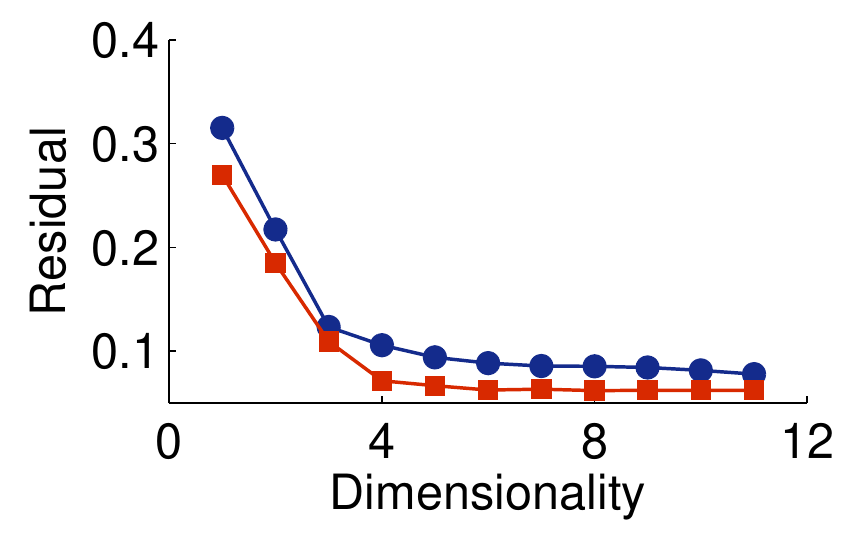}
        	\caption*{(b)}       	
    	\end{subfigure}        	
        	\caption{An abrupt phase change of the crowd behavior where a walking crowd suddenly starts running \cite{detection}. (a) The first phase of the motion (walking) is embedded onto the blue colored manifold, while the second phase (running) is embedded onto the red colored manifold. Two snapshots showing walking and running at time steps $t_1$ and $t_2$ are embedded onto points $\bs{p}^{(t_1)}$ and $\bs{p}^{(t_2)}$, respectively, in the corresponding manifolds. The locus of singularities ($\mathcal{L}$) is represented by orange color.  (b) The scaled residual variance with respect to the dimensionality, which the dimensionality of the underlying manifold is given by an elbow, is obtained by running Isomap upon frames in each phase with 6 nearest neighbors. Embedding dimensionalities of sub-manifolds representing walking (blue circle) and running (red square) of the crowd are three and four respectively. 
}\label{fig:crowd_walk_run}
\end{figure}

A phase transition in a multi-agent system is defined as a switching of the current smooth embedding manifold into another smooth manifold with different  dimensionality as trajectories of agents evolve in the phase space. Our approach of detecting phase transitions is based on revealing high curvature on the manifold which differentiate phases of the motion. We hypothesize that a phase transition is manifested in the form of change in local curvature that can be detected by a ratio of singular values computed on points sampled on the manifold. We first formulate this concept in two dimensions by proving a relation between curvature and singular value ratio in a curve, and then extend it to higher-dimensions by using the shape operator. Then, we justify that the same phase transition can also be detected by analyzing the ratios of the smallest singular value to the largest singular value which are computed locally upon neighborhoods of points sampled on the manifold. Based on the distribution of a moving sum of absolute moving difference of the singular value ratios, phase transitions are detected and their magnitudes are ordered.

This paper is organized as follows: Section ~\ref{sec:method} describes the method of detecting phase transitions. In Section ~\ref{sec:algorithm}, the method along with the detailed algorithm of detecting phase transitions in collective behavior is presented. Section ~\ref{sec:examples} describes the performance of the method using one synthetic dynamical simulation of the Vicsek model and three natural experimental data sets, a crowd of human, a flock of birds, and a school of fish. We conclude the work in Section ~\ref{sec:conclusion_discussion} with a discussion of the method including the performance and future work.

\section{Method of detecting phase transitions}\label{sec:method}
We declare that, in presence of a phase transition, the curvature of the underlying manifold is abruptly changed. Therein, we first approximate the point-wise curvature of a curve in two dimensions and then extend it to higher dimensions using the shape operator. Finally, we attest that the same phase transition can also be approximated locally by singular value ratios computed over the data sampled on the manifold. 

\subsection{Approximating the curvature of a curve}
Point-wise curvature of a curve is approximated  by using singular value ratio computed over the data sampled in a neighborhood at the point. We assume that the data is evenly distributed along the curve. As in Figure ~\ref{fig:trans_circle}(a), we can superimpose a neighborhood at any given point $\bs{p}$ on the curve with an arc $\bs{p}_1\bs{p}\bs{p}_2$ of a translating circle such that the arc subtends a small angle of $2T$ at the center $\bs{O}$. We consider the curvature of a translating circle instead of that of the curve since they are similar in aforesaid setup. Thus, we prove the assertion providing the relationship between curvature and  singular value ratios over the data sampled on a circle. 

\begin{figure}[hpt]
        	\centering
	\begin{subfigure}[b]{0.52\textwidth}
        	\includegraphics[width=1\textwidth]{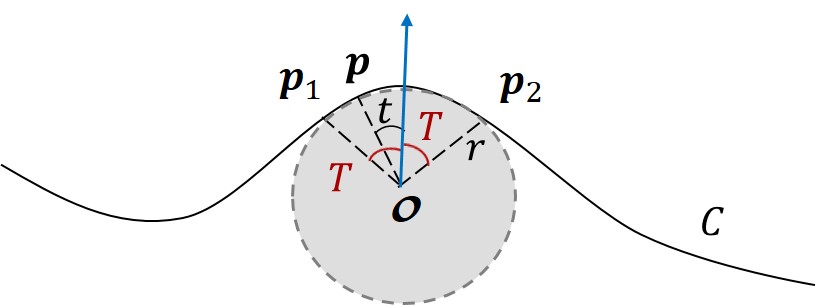}
        	\caption*{(a)}        	
    	\end{subfigure}   
	\quad
    	\begin{subfigure}[b]{0.44\textwidth}
	\vspace{20pt}
        	\includegraphics[width=1\textwidth]{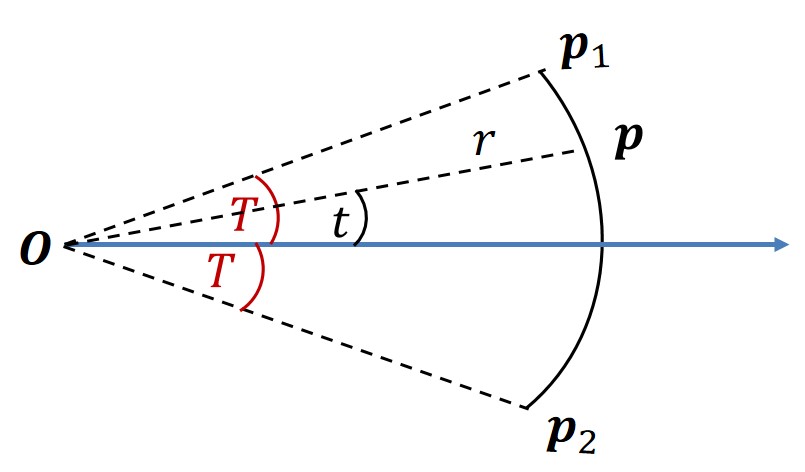}
        	\caption*{(b)}       	
    	\end{subfigure}        	
        	\caption{(a) Superimposing a neighborhood of the curve $C$ at point $\bs{p}$ with an arc  $\bs{p}_1\bs{p}\bs{p}_2$ which subtend an small angle of $2T$ at the origin of a translating circle. (b) Zoomed and rotated secular sector $\bs{p}_1\bs{p}_2\bs{O}$, in (a) such that the blue arrow is horizontal.}
\label{fig:trans_circle}
\end{figure}

\begin{theorem} \label{thm:curvature}
Given a circle centered at the origin with radius $r$, let $\alpha$ number of points are uniformly distributed with density $\rho$ on an arc which subtends an angle of $2T$ at the center. Let $\sigma_1$ and $\sigma_2$ ($\sigma_1 > \sigma_2$), are two singular values computed upon the data on the curve, then $\sigma_2/\sigma_1 \approx M\kappa$ where $\kappa=1/r$ and $M=\frac{\alpha}{2\sqrt{15}\rho} \in \mathbb{R}^+$.
\end{theorem}

\begin{proof} 
We rotate the circular sector $\bs{p}_1\bs{p}_2\bs{O}$ such that the line bisecting the angle $\bs{p}_1\bs{p}\bs{p}_2$ which is shown by a blue arrow in Figure ~\ref{fig:trans_circle}(a) is horizontal. We consider this blue arrow as the horizontal axis and compute the singular values of the data sampled from the arc. A point $\bs{p}(x_1,x_2)$ on the arc is given by $x_1=r \cos{t}$ , $x_2=r \sin{t}$ for $t\in[-T,T]$. Expected values of variables $x_1$ and $x_2$ are computed as
\begin{equation}\mu_{x_1} =\frac{1}{2T}\int_{-T}^{T} r\cos{t} \ dt = \frac{1}{T}r \sin{T}\end{equation} and
\begin{equation}\mu_{x_2}=\frac{1}{2T}\int_{-T}^{T} r\sin{t} \ dt = 0,\end{equation} respectively. We then compute pairwise covariance between variables as
\begin{equation}cov(x_1,x_1)=\frac{1}{2T}\int_{-T}^{T} (r \cos{t}-\mu_{x_1})^2dt=\frac{r^2}{2T}\left(T+\frac{1}{2}\sin{2T}+\frac{2}{T} \sin^2{T}\right),\end{equation}
\begin{equation}cov(x_1,x_2)=\frac{1}{2T}\int_{-T}^{T} (r \cos{t}-\mu_{x_1})(r \sin{t}-\mu_{x_2})dt=0,\end{equation}
\begin{equation}cov(x_2,x_1)=\frac{1}{2T}\int_{-T}^{T} (r \sin{t}-\mu_{x_2})(r \cos{t}-\mu_{x_1})dt=0,\end{equation} and 
\begin{equation}cov(x_2,x_2)=\frac{1}{2T}\int_{-T}^{T} (r \sin{t}-\mu_{x_2})^2dt=\frac{r^2}{4T}(2T-\sin{2T}).\end{equation}

The covariance matrix $\Sigma$, of the data is
\begin{equation}\Sigma
=\left(\begin{matrix} \frac{r^2}{2T}\Big(T+\frac{1}{2}\sin{2T}+\frac{2}{T} \sin^2{T}\Big) && 0 \\ 0 && \frac{r^2}{4T}(2T-\sin{2T}) \end{matrix}\right).
\end{equation}
The covariance matrix which is approximated to $5^{th}$ order of $T$ by using the Taylor's expansion is
\begin{equation}\tilde{\Sigma} = \left(\begin{matrix}
\frac{r^2 T^4}{45} && 0 \\ 0 && \frac{r^2 T^2}{3}\left(1-\frac{T^2}{5}\right)
\end{matrix}\right)+O (T^6).
\end{equation} 
Let eigenvalues of $\tilde{\Sigma}$ are $\lambda_1$ and $\lambda_2$ ($\lambda_1 > \lambda_2$), then
\begin{equation}\lambda_1\approx \frac{r^2 T^2}{3}\left(1-\frac{T^2}{5}\right) \ \text{and} \ \ \lambda_2\approx \frac{r^2 T^4}{45}.
\end{equation} 

As eigenvalues are computed upon the covariance matrix, they also relate to principal components. We denote singular values associated with principal components by $\sigma_1$ and $\sigma_2$ ($\sigma_1 > \sigma_2$) and relate them to eigenvalues as $\sigma_2/\sigma_1=\sqrt{\lambda_2/\lambda_1}$ \cite{gerbrands1981}. Then, 
\begin{equation}\frac{\lambda_2}{\lambda_1}\approx \frac{T^2}{3(5-T^2)} \implies \frac{\sigma_2}{\sigma_1}\approx \frac{T}{\sqrt{3(5-T^2)}}.
\end{equation}
Further,
\begin{equation}\label{eqn:sigma_ratio}\frac{\sigma_2}{\sigma_1} \approx \frac{T}{\sqrt{15}}\end{equation}
since $T$ is small.
Let, $\alpha$ points are uniformly distributed with the density $\rho$ on the arc, then,
\begin{equation}\label{eqn:value_T}
T=\frac{\alpha}{2\rho}\kappa \ , \ \text{where} \ \kappa=\frac{1}{r}.
\end{equation}
By (\ref{eqn:sigma_ratio}) and (\ref{eqn:value_T}),
\begin{equation}
\frac{\sigma_2}{\sigma_1} \approx M \kappa \ \text{where} \ M = \frac{\alpha}{2\sqrt{15}\rho} \in \mathbb{R}^+.
\end{equation}
\qedhere
\end{proof}

\subsection{Approximating the curvature of a manifold}\label{sub:curvature_mani}
We extend the two dimensional assertion into higher dimensions by intersecting principal sections, made by the shape operator, with the manifold. Extrinsic curvatures of a manifold in orthogonal tangential directions are measured using eigenvalues and eigenvectors of the shape operator \cite{o1966elementary,  rovenskii2011topics}, such that eigenvalues provide magnitudes and eigenvectors provide directions of principal curvatures along tangential directions \cite{alfred1998modern}. Bellow we explain the computation of the shape operator and principal sections.

Let $\mathcal{M}^m=\big(f_1(x_1, \dots, x_m), \dots,  f_m(x_1, \dots, x_m)\big)$ represents the parametric form of the manifold embedded in the Euclidean space $\mathbb{R}^{m+1}$. For $j=1, \dots, m$,
\begin{equation}
\bs{v}_{\bs{p}}^{(j)}=\frac{\partial \mathcal{M}^m}{\partial x_j} \hspace{20pt} \text{and} \hspace{20pt} \hat{\bs{v}}_{\bs{p}}^{(j)}=\frac{\bs{v}_{\bs{p}}^{(j)}}{\big\|\bs{v}_{\bs{p}}^{(j)}\big\|}
\end{equation}
provide mutually orthogonal tangential vectors and unit tangential vectors of $\mathcal{M}^m$ at $\bs{p}$, respectively. Thus, $\big\{\hat{\bs{v}}_{\bs{p}}^{(j)}\big| \ \forall \  j\big\}$ is an orthonormal basis for the tangential space at $\bs{p}$. The shape operator $\mathcal{S}_{\bs{p}}$, of the manifold $\mathcal{M}^{m}$ at the point $\bs{p}$ is defined as
\begin{equation}
\mathcal{S}_{\bs{p}}=\big(-\nabla_{x_{j_1}} N_{\bs{p}} \cdot \hat{\bs{v}}_{\bs{p}}^{(j_2)}\big)_{j_1,j_2} \hspace{20pt} \text{for} \hspace{20pt}  j_1, j_2=1, \dots, m
\end{equation}
 \cite{o1966elementary,  rovenskii2011topics}. 

Let \big($\kappa\big(\bs{u}_{\bs{p}}^{(j)}\big), \bs{u}_{\bs{p}}^{(j)}$\big) for $j=1, \dots, m$ denote eigenpairs  of $\mathcal{S}_{\bs{p}}$, then magnitudes and directions of principal curvatures are given by $\kappa\big(\bs{u}_{\bs{p}}^{(j)}\big)$ and $\bs{u}_{\bs{p}}^{(j)}$, respectively \cite{alfred1998modern}. The two-dimensional plane in $\mathbb{R}^{m+1}$ which is spanned by the principal direction $\bs{u}^{(j)}_{\bs{p}}$ and the unit normal at $\bs{p}$, $N_{\bs{p}}$, is defined as the $j$-th principal section,
\begin{equation}
\Pi_{\bs{p}}^{(j)}=\left\{\beta \bs{u}_{\bs{p}}^{(j)} + \gamma N_{\bs{p}} \vert \beta, \gamma \in \mathbb{R}\right\}.
\end{equation}
Thus, for $j=1, \dots, m$, $\Pi_{\bs{p}}^{(j)}$ are mutually orthogonal planes. A computational example of producing principal sections is attached in Appendix ~\ref{app:shape_operator}. Intersection of $\Pi_{\bs{p}}^{(j)}$'s with the manifold makes curves $C^{(j)}\in \mathbb{R}^{m+1}$ such that each passes through the point $\bs{p}$. Figure ~\ref{fig:curvature_local} illustrates principal sections and curves made through that when $m=2$.

\begin{figure}
        \centering
        \includegraphics[width=.75\textwidth]{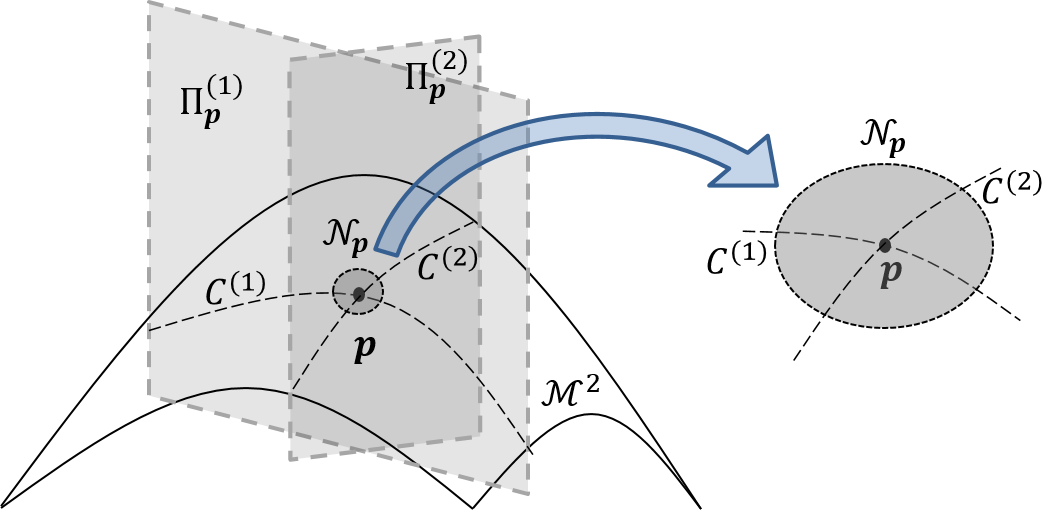}
        \caption{Local distribution of data around the point $\bs{p}$ on a two dimensional manifold ($\mathcal{M}^2$). Principal sections $\Pi^{(1)}_{\bs{p}}$ and $\Pi^{(2)}_{\bs{p}}$ are created by using the shape operator at $\bs{p}$, and curves $C^{(1)}$ and $C^{(2)}$ are produced by intersecting $\Pi^{(1)}_{\bs{p}}$ and $\Pi^{(2)}_{\bs{p}}$ with $\mathcal{M}^2$, respectively.}
\label{fig:curvature_local}
\end{figure}

Without loss of generality, for $j=1, \dots, m$, we assume that the data is distributed uniformly with a sufficient density  on $C^{(j)}$ to contain $\alpha$ points  in a neighborhood at the point $\bs{p}$. For all $j$, singular values $\sigma^{(j)}_1$ and $\sigma^{(j)}_2$ ($\sigma^{(j)}_1 > \sigma^{(j)}_2$) are computed upon the data sampled in the neighborhood of $\bs{p}$ on $C^{(j)}$. By Theorem ~\ref{thm:curvature}, curvature $\kappa^{(j)}$ on the curve $C^{(j)}$ is approximated as $\sigma^{(j)}_2/\sigma^{(j)}_1 \approx M^{(j)}\kappa^{(j)}$ for some $M^{(j)}\in \mathbb{R}^+$. A phase transition differentiates the manifold into two sub-manifolds such that each map with a different Euclidean space \cite{lee2012introduction}. Thus, under a phase transition, geometry permits that the curvature of some curves, $C^{(j)}$'s,  undergo abrupt changes which also result abrupt changes of the ratios $\sigma^{(j)}_2/\sigma^{(j)}_1$. 

$\textbf{A generic example.}$ As a simple geometric example to illustrate high curvature on a manifold at a phase transition, we use a three-dimensional joined-manifold that we call a `sombrero-hat'  (Fig. ~\ref{fig:example_sombrero}(a)) of 2000 points produced by the equations 
\begin{equation}\label{eqn:sombrero}
\begin{split}
x_1=R \ cos\theta, \\
x_2=R \ sin\theta, \\
\text{and} \ x_3 =
\begin{cases}
4-R^2, & \text{if }R\le2 \\
0, & \text{if }R>2
\end{cases}
\end{split}
\end{equation}
for $R\in \mathbb{U} [0,4]$ and $\theta\in \mathbb{U} [0,\pi]$. This sombrero-hat intersects two sub-manifold, brim (green) and crown (blue) at the locus of singularities (red) representing a phase transition. Instead of constructing principal sections using the shape operator, for simplicity, we intuitively find a principal section in this example at the point $(0, 0, 4)$. Since the curvature of the manifold is same in all tangential direction at this point, we choose one principal direction to be $\hat{\bs{i}}$, the unit vector along the $x_1$-axis. The unit normal at this point is $\hat{\bs{k}}$ which is the unit vector along the $x_3$-axis. Thus, we define the plane $\{\beta_1\hat{\bs{i}}+\beta_2\hat{\bs{k}} \vert \beta_1, \beta_2 \in \mathbb{R}\}$ as the principal section. Intersection of this principal section with the sombrero-hat gives a curve as shown in Figure ~\ref{fig:example_sombrero}(b) which has a high curvature at the red points. Isomap residual plots indicate different embedding dimensionalities at the locus than those of two sub-manifolds.
\begin{figure}[t]
        	 \centering
        	\includegraphics[width=1\textwidth]{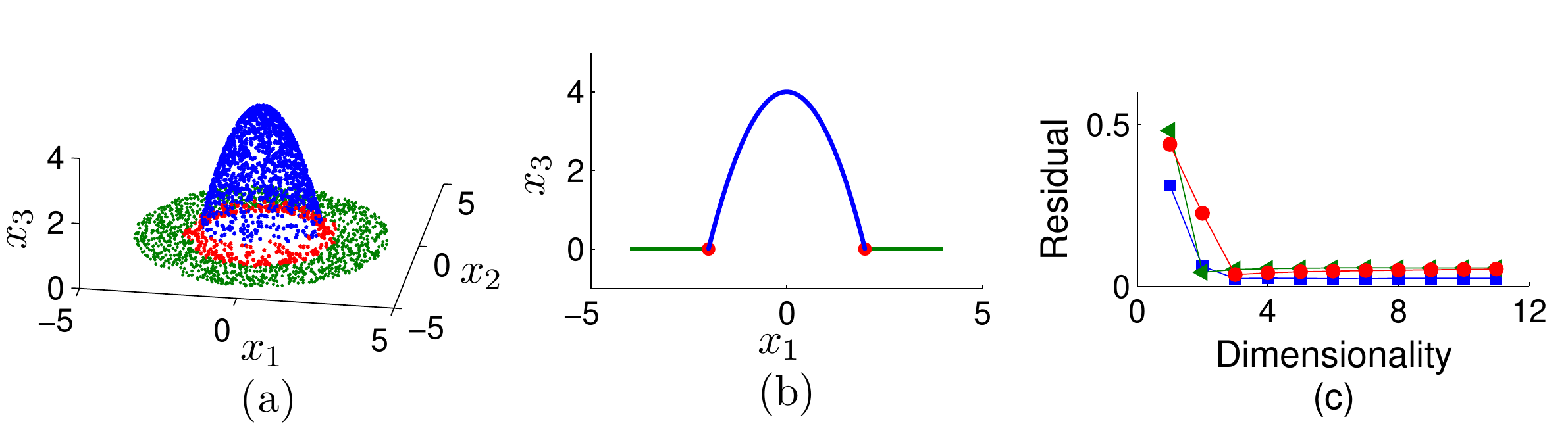}
        	\caption{(a) A three dimensional  sombrero-hat of 2000 points consisting two sub-manifolds (blue and green) and locus of singularities (red) is intersected with the plane $\{\beta_1\hat{\bs{i}}+\beta_2\hat{\bs{k}} \vert \beta_1, \beta_2 \in \mathbb{R}\}$ to produce (b) a curve in $\mathbb{R}^3$. (c) Isomap residual plots those show embedding dimensionalities by elbows reveal that the dimensionalities of two sub-manifolds are two while the dimensionality of the locus is three.}
\label{fig:example_sombrero}
\end{figure}

\subsection{Phase transition via local data distribution on the manifold}\label{sub:local_approximation}
As the construction of the shape operator at a neighborhood of each point on the manifold is computationally expensive, here we present an alternative approach to compute the singular value ratios and detect phase transitions. Therein, we first make a neighborhood  $\mathcal{N}_{\bs{p}}\in\mathbb{R}^{m+1}$ at each point $\bs{p}$, such that it contains $\alpha \in \mathbb{N}$ points by using nearest neighbor search algorithm given in \cite{friedman1977algorithm, yianilos1993data}. Then, we perform singular value decomposition\footnote{Singular value decomposition finds singular values $\sigma_1, \dots, \sigma_{\alpha}$ for some $\alpha \in \mathbb{N}$, and two unitary matrices $U$ and $V$, such that $U^{'}U=V^{'}V=I_{\alpha}$, those provide the decomposition $\mathcal{D}^{\alpha}_n=U\Sigma V^{'}$ for $\Sigma =\text{diag}(\sigma_1, \sigma_2, \dots \sigma_{\alpha})$ with $\sigma_1>\dots > \sigma_{\alpha}$ \cite{golub1970singular}.}
for data in each $\mathcal{N}_{\bs{p}}$ and denote the descending order by $\sigma_1, \dots, \sigma_\alpha$ for some $\alpha \in \mathbb{N}$. Without loss of generality, we assume that the data distribution in each $C^{(j)}$ is dense enough for $\mathcal{N}_{\bs{p}}$ to contain data from each curve $C^{(j)}$, $j=1, \dots, m$, made in Section ~\ref{sub:curvature_mani}. However, $\mathcal{N}_{\bs{p}}$ contains at least the point $\bs{p}$, from each $C^{(j)}$ as shown in Figure ~\ref{fig:curvature_local} which depicts the same scenario for a two dimensional manifold $\mathcal{M}^2$, in $\mathbb{R}^3$. 

According to this setting, we assert that an abrupt change which is detected by the ratio $\sigma^{(j)}_2/\sigma^{(j)}_1$ in some curves created through $\bs{p}$ is also detected by the ratio $\sigma_\alpha/\sigma_1$ computed over the data  in $\mathcal{N}_{\bs{p}}$ on the manifold $\mathcal{M}^m$. In this study, we use frames partitioned from a collective motion video as data. A group of agents evolving with mutual interaction can be explained in terms of a hidden manifold structure such that each frame corresponds to a single data point on the manifold. Thus, we consider the point $\bs{p}$ on the manifold corresponds to the $n$-th frame of the video. Let consider $N$ frames are embedded on the manifold, thus the distribution
\begin{equation}\label{eqn:sigma_distribution}
\{(\sigma_{\alpha}/\sigma_1)_{n} \ \vert \ n=1, \dots, N\},
\end{equation}
provides the magnitudes of the phase at each frame. Phase changes between consecutive frames are measured by moving differences of singular value ratios computed over neighborhoods of corresponding points as
\begin{equation}
t_n=(\sigma_{\alpha}/\sigma_1)_{n+1}-(\sigma_{\alpha}/\sigma_1)_{n} \ \ \text{for} \  \ n=1, 2, \dots, N-1. 
\end{equation}

As the distribution $t_n$ is highly volatile, we utilize $\alpha$-points moving sum
\begin{equation}
\Sigma_n^{\alpha}=\sum_{i\in\left[n-\ceil{\frac{\alpha}{2}}, n+\ceil{\frac{\alpha}{2}}\right] \cap \mathbb{N}} t_i \hspace{10pt} ; \hspace{10pt} n\in \left[\Big\lceil\frac{\alpha}{2}\Big\rceil, N-1-\Big\lceil\frac{\alpha}{2}\Big\rceil\right]\cap \mathbb{N},
\end{equation}
where $\ceil{\alpha/2}$ is the ceiling function\footnote{Ceiling function of $x$, denoted by $\ceil{x}$, is the largest integer less than or equal to $x$.}, and smooth the distribution. After a phase transition occurs, manifold curvature is abruptly changed and then so the singular values. Thus, a phase transition between $n$-th and $(n+1)$-th frames can be quantified by the magnitude of $\Sigma_n^{\alpha}$.
 
\section{Algorithm for detecting phase transitions}\label{sec:algorithm}
This algorithm requires two inputs, one is the parameter $\alpha$ for number of nearest neighbors and the other is the data matrix $\mathcal{D}$ constructed as bellow. We assume input data is gray-scale images partitioned from a video of collective behavior.  We reshape each frame into a row matrix and produce the data matrix $\mathcal{D}$ by concatenating each row matrix vertically such that $n$-th row in $\mathcal{D}$ represents the pixel intensities of the $n$-th frame for some $n$ \cite{bovik2010handbook}. Since the $n$-th row of $\mathcal{D}$, denoted by $\mathcal{D}(n,:)$, is the coordinates of the $n$-th point on the manifold, Euclidean distance, denoted by $d_{\mathcal{D}}$, between any two points $n$ and $n'$, on the manifold is computed as 
\begin{equation}
d_{\mathcal{D}}(n, n')=\|\mathcal{D}(n,:)-\mathcal{D}(n',:)\| \ ; \ n, n' = 1, \dots, N.
\end{equation}
Based on this distance, $\alpha$ nearest neighbors for the $n$-th point are extracted and denoted as the set $\mathcal{N}_n$ (same as $\mathcal{N}_{\bs{p}}$ in Section ~\ref{sub:local_approximation}). Then, we compute $\sum^\alpha_n$ over $\mathcal{N}_n$ as explained in the Section ~\ref{sub:local_approximation}. Algorithm outputs magnitudes of phase changes $\Sigma^{\alpha}_n$, so we can choose the largest among them as phase transitions. The method of detecting phase transitions is given as Algorithm ~\ref{alg:detecting_transitions}.

\begin{algorithm*}
\caption{ \textit{Phase Transition Detection in Collective Behavior}}\label{alg:detecting_transitions}
\begin{algorithmic}[1]
\Statex{\textbf{Procedure} \textit{PTD} ($\alpha$, $\mathcal{D}$)}
\State Perform nearest neighbor search in \cite{friedman1977algorithm, yianilos1993data} to obtain $\alpha$ nearest neighbors for each frame in $\mathcal{D}$. Here, we denote $\alpha$ nearest neighbors of the $n$-th frame in $\mathcal{D}$ by the set $\mathcal{N}_n$.
\State Perform singular value decomposition over \cite{golub1970singular} on $\mathcal{N}_n$ and denote the descending order of singular values by $\sigma_1, \dots, \sigma_{\alpha}$ for some $\alpha \in \mathbb{N}$.
\State Compute the ratio $\sigma_{\alpha}/\sigma_1$ for $\mathcal{N}_n$ and denote by $(\sigma_{\alpha}/\sigma_1)_n$ where $n=1,\dots, N$.
\State Compute absolute moving difference of ratios, $t_n=|(\sigma_{\alpha}/\sigma_1)_{n+1}-(\sigma_{\alpha}/\sigma_1)_{n}|$, for $n=1,\dots,N-1$.
\State Compute $\alpha-$points moving sum, $\Sigma_n^{\alpha}=\sum_{i\in\left[n-\ceil{\alpha/2}, n+\ceil{\alpha/2}\right] \cap \mathbb{N}} t_i$, of $t_i$ for $n\in[\ceil{\alpha/2}, N-1-\ceil{\alpha/2}]\cap \mathbb{N}$, where $\ceil{\alpha/2}$ is the ceiling function.
\State Largest values of $\Sigma_n^{\alpha}$ are extracted as phase transitions.
\Statex{\textbf{end procedure}}
\end{algorithmic}
\end{algorithm*}

\section{Group behavioral examples}\label{sec:examples}
In this section, we evaluate the phase transition detection algorithm on four datasets: a synthetic dynamical simulation of the Vicsek model and three natural data sets; a crowd of people, a flock of birds, and a school of fish.

\subsection{A simulation of the Vicsek model} \label{sec:example_vicsek}
A swarm of self propelled particles with three imposed phase transitions simulated by the Vicsek model \cite{vicsek1995novel} is examined to investigate the sensitivity of the algorithm. This model outputs two-dimensional positions and orientations of particles in each time-step.

We modify the notation for set of nearest neighbors used in Section ~\ref{sub:local_approximation} as $\mathcal{N}_{n}^{(i)}$ to denote all the nearest-neighbors of the $i$-th particle within a unit distance at the $n$-th time-step. The Vicsek model \cite{vicsek1995novel} updates the orientation $\theta_{n}^{(i)}$ of the $i$-th particle at the $n$-th time step as
\begin{equation}\label{eqn:vicsek1}
\begin{split}
\theta_{n}^{(i)}=\arg{\left(\bs{V}_{n-1}^{(i)}\right)}+\epsilon_{n-1}^{(i)},\\
\text{where} \ \bs{V}_{n-1}^{(i)}=\frac{1}{\big|\mathcal{N}_{n-1}^{(i)}\big|} \sum_{j\in \mathcal{N}_{n-1}^{(i)}} \begin{bmatrix}\cos\left(\theta_{n-1}^{(j)}\right) \\ \sin\left(\theta_{n-1}^{(j)}\right)\end{bmatrix}
\end{split}
\end{equation}
and $\epsilon_{n-1}^{(i)}$ is the noise parameter sampled from a Gaussian distribution with mean zero and standard deviation $\sigma$. Here, $\bs{V}_{n-1}^{(i)}$ is the average direction of motion of all the particles in $\mathcal{N}_{n-1}^{(i)}$. The position $\bs{a}_{n}^{(i)}$ of the $i$-th particle at the $n$-th time step is therefore updated as
\begin{equation}\label{eqn:vicsek2}
\bs{a}_{n}^{(i)}=\bs{a}_{n-1}^{(i)}+s_{n-1}^{(i)} \begin{bmatrix}\cos\left(\theta_{n-1}^{(i)}\right) \\ \sin\left(\theta_{n-1}^{(i)}\right)\end{bmatrix} \delta, 
\end{equation} 
where $\delta$ is the time-step size and $s_{n-1}^{(i)}$ is the speed of the $i$-th particle at the time-step ($n-$1).

We run a simulation of this model and generate a synthetic data set of a particle swarm of 50 particles in 200 time-steps with periodic boundary conditions. We set $\delta=0.05$ and $s_{n}^{(i)}=0.1$ for all $i$ and $n$. We impose three phase transitions by changing the noise $\epsilon_{n}^{(i)}$, using a variable standard deviation
\begin{equation}
\sigma =
\begin{cases}
0.25, & \text{if }n<50, \\
1, & \text{if }51 \le n < 100,\\
0.05, & \text{if }101 \le n < 150,\\
0.75, & \text{if }n\ge 150,
\end{cases}
\end{equation}
for all $i$, in the Gaussian distribution. For $n=1, \dots, 200$, we convert point-mass positions of particles  at the $n$-th time-step, $\big\{\bs{a}_{n}^{(i)}\big\vert i=1, \dots, 50\big\}$, into a gray-colored frame such that the position of each particle is represented by a black square of pixels $10\times 10$. Therein, we first convert the positions of particles into a sparse matrix $A_n$, at each time step $n=1, \dots, 200$. Then we make a rotationally symmetric Gaussian low-pass filter $B$ of size $10\times 10$ pixels with standard deviation of 10 pixels and filter the matrix $A_n$ by $B$ to generate
\begin{equation}
C_n(n_1, n_2)=\sum^{10}_{k_1=1}\sum^{10}_{k_2=1}A_n(n_1-k_1, n_2-k_2)B(k_1, k_2) \ ; \ \forall \ n_1, n_2
\end{equation}
for $n=1, \dots, 200$ \cite{bracewell2003fourier}. Each matrix $C_n$ is converted into a frame and all the frames are finally converted into a one data matrix $\mathcal{D}$.

We run the phase transition detection algorithm upon $\mathcal{D}$ with several $\alpha$ values and observe that the distribution  $(\sigma_4/\sigma_1)_{n}$ generated by $\alpha=4$ (Fig. ~\ref{fig:example_viksec}(a)) shows clear phase changes with respect to noise changes. The plot of 20 largest phase changes (Fig. ~\ref{fig:example_viksec}(b)) reveals that three phase changes at frame numbers $150, 99$ and $50$ are significant than others and we consider them as phase transitions. Thus, the manifold representing the whole motion is partitioned into phase transition free 4 sub-manifolds ranging $1-50, 51-99, 100-150$ and $151-200$. We run Isomap on each range of frames with the neighborhood parameter four and obtain the residual plots given in Figure ~\ref{fig:example_viksec}(c). The Isomap residual plots indicating the embedding dimensionality by an elbow, reveal that the sub-manifolds are embedded in three, six, two, and four dimensions, respectively.

\begin{figure}[tp]
        	 \centering
        	\includegraphics[width=1\textwidth]{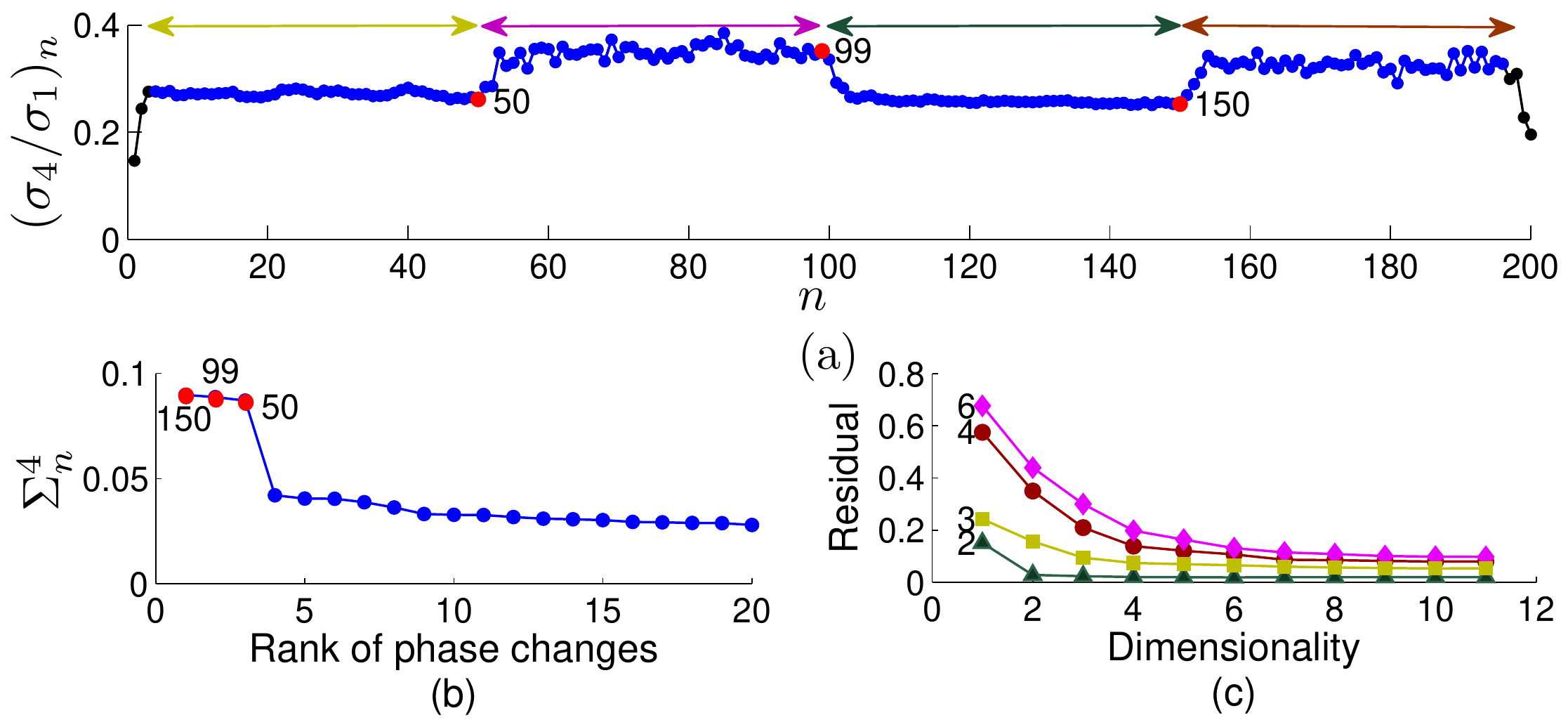}
        	\caption{Detecting phase transitions in a particle swarm simulated using the Vicsek model with alternating noise levels. (a) The distribution of $(\sigma_4/\sigma_1)_{n}$  versus frame numbers. Therein, the range of frames for each sub-manifold is represented by a left-right arrow and the frame number at each phase transition is represented by a red circle along with the frame number associated. (b) The plot of 20 largest phase changes including frame numbers of three phase transitions. (c) Isomap residual variance versus dimensionality of each sub-manifold.
}\label{fig:example_viksec} 
\end{figure}

\subsection{A human crowd}\label{sec:example_crowd}
A video of a human crowd available on-line at \cite{pet2009} containing a phase transition at the $69$-th frame from walk to run is considered now.

As in Example ~\ref{sec:example_vicsek}, we first convert all the frames into a data matrix $\mathcal{D}$, and then run phase transition detection algorithm on it. When $\alpha=3$, we observe that the distribution ($\sigma_3/\sigma_1)_n$ given in Figure ~\ref{fig:example_crowd}(a) shows clear trade-off about the frame numbers. Figure ~\ref{fig:example_crowd}(b) shows that only the first phase transition having a magnitude of $0.23$ is significant. The manifold representing the crowd is partitioned such that frames $1-69$ represent one sub-manifold and frames $70-90$ represent the other sub-manifold. The snapshots in Figure ~\ref{fig:example_crowd}(a) show the phases walking (left) and running (right) of the crowd. Isomap running on frames $1-69$ and $70-90$ with $\alpha=3$ reveals that the embedding dimensionalities of corresponding sub-manifolds are three and four, respectively. 

\begin{figure}[hpt]
\centering
\includegraphics[width=1\textwidth]{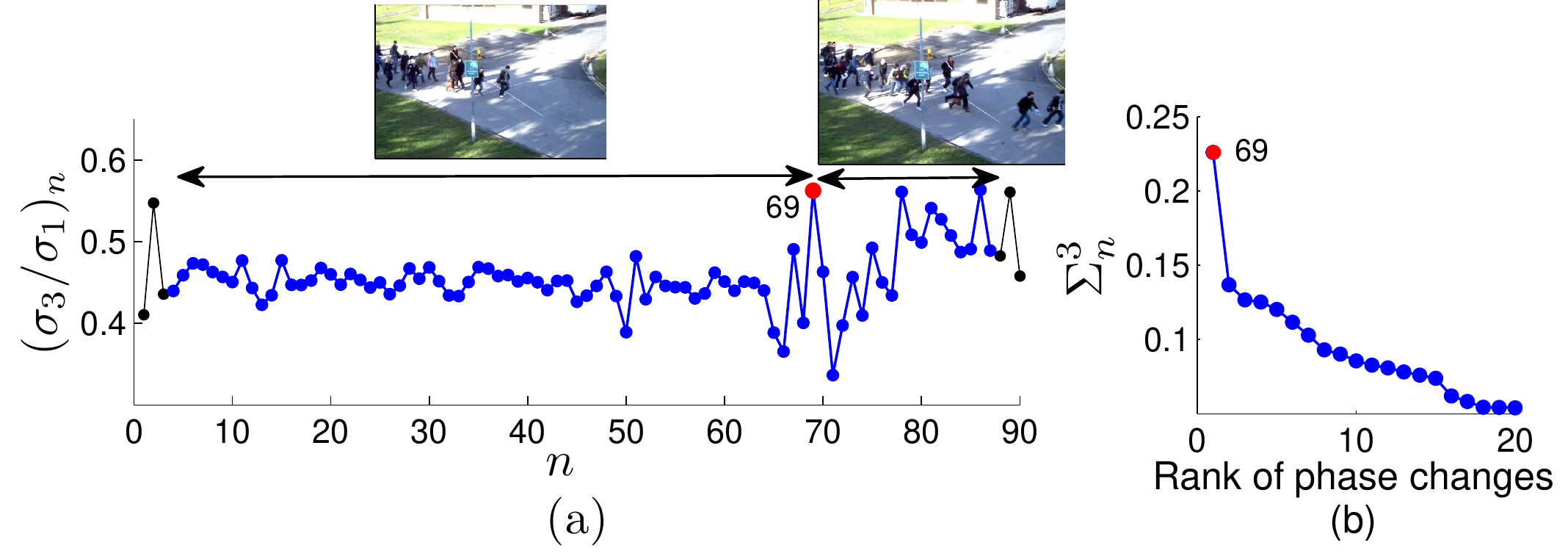}
\caption{Detecting a phase transition between phases of walking and running in a human crowd \cite{pet2009}. (a) The distribution of $(\sigma_3/\sigma_1)_{n}$ versus frame numbers. Therein, while the snapshots show instances of the crowd in each phase, left-right arrows and the red circle represent ranges of frames in each sub-manifold and the frame at the phase transition, respectively. (b) The plot of 20 largest phase changes representing the phase transition in red along with its frame number. 
}\label{fig:example_crowd} 
\end{figure}

\subsection{A bird flock}\label{sec:example_birds}
We now detect a phase transition, differentiating phases of sitting and flying, at the $58$-th frame in a video of a bird flock obtained on-line at \cite{birds2011}. 

We run phase transition detection algorithm on the data matrix with $\alpha=3$ and obtain the distribution ($\sigma_3/\sigma_1)_n$ in Figure ~\ref{fig:example_birds}(a). The plot of 20 largest phase changes in Figure ~\ref {fig:example_birds}(b) illustrates that the video consists one phase transition  at the $58$-th frame and has the magnitude of $0.32$. Then, the manifold is partitioned into two sub-manifolds representing frames in the ranges $1-58$ and $59-100$. Isomap running on aforesaid ranges with $\alpha=3$ reveals that the embedding dimensionalities of corresponding sub-manifolds are two and three, respectively.

\begin{figure}[hpt]
\centering      	
\includegraphics[width=1\textwidth]{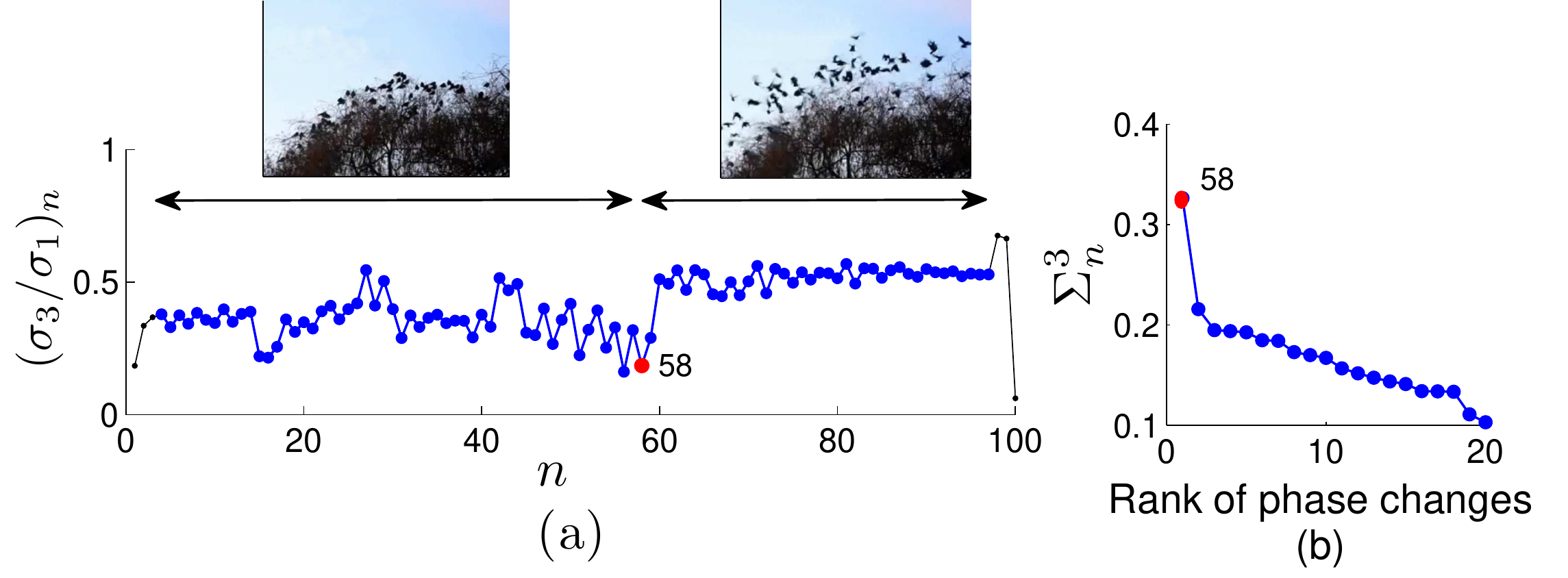}      
\caption{Detecting a transition in a bird flock between phases sitting and flying \cite{birds2011}. (a) The distribution of $(\sigma_3/\sigma_1)_{n}$ shows ranges of frames representing two sub-manifolds by left-right arrows and instances phases  of the flock by snapshots. (b) The plot of 20 largest phase changes. The frame at the phase transition is represented by red in Figures (a) and (b).}
\label{fig:example_birds} 
\end{figure}

\subsection{A fish school}\label{sec:example_fish}
Finally, we use a video of a fish school having few phase transitions to validate the method. This school is stimulated with panics and the behavior is recorded.

We execute the phase transition detection algorithm on the data matrix with $\alpha=6$ and obtain the distribution of $(\sigma_6/\sigma_1)_{n}$ (Fig. ~\ref{fig:example_fish}(a)). The plot of phase changes (Fig. ~\ref{fig:example_fish}(c)) reveals that the first four phase changes located at frames $40, 112, 185$ and $222$ are phase transitions. Now, the manifold is partitioned into phase transition free sub-manifolds representing frames in ranges $1-40, 41-112, 113-185, 186-222$ and $223-250$ as shown by left-right arrows in Figure ~\ref{fig:example_fish}(a). According to Figure ~\ref{fig:example_fish}(b), the pair of snapshots at the 40-th frame exhibits a large global abrupt change since the school reacts to the panic together. However, by the evidence observed through Figure ~\ref{fig:example_fish}(b) and snapshots,  the second and the third phase transitions are abrupt and local while the last is gradual and local. Isomap is run on all sub-manifolds with $\alpha=6$ and embedding dimensionalities are obtained as two, four, three, two, and six, respectively.

\begin{figure}[hpt]
\centering            	
\includegraphics[width=1\textwidth]{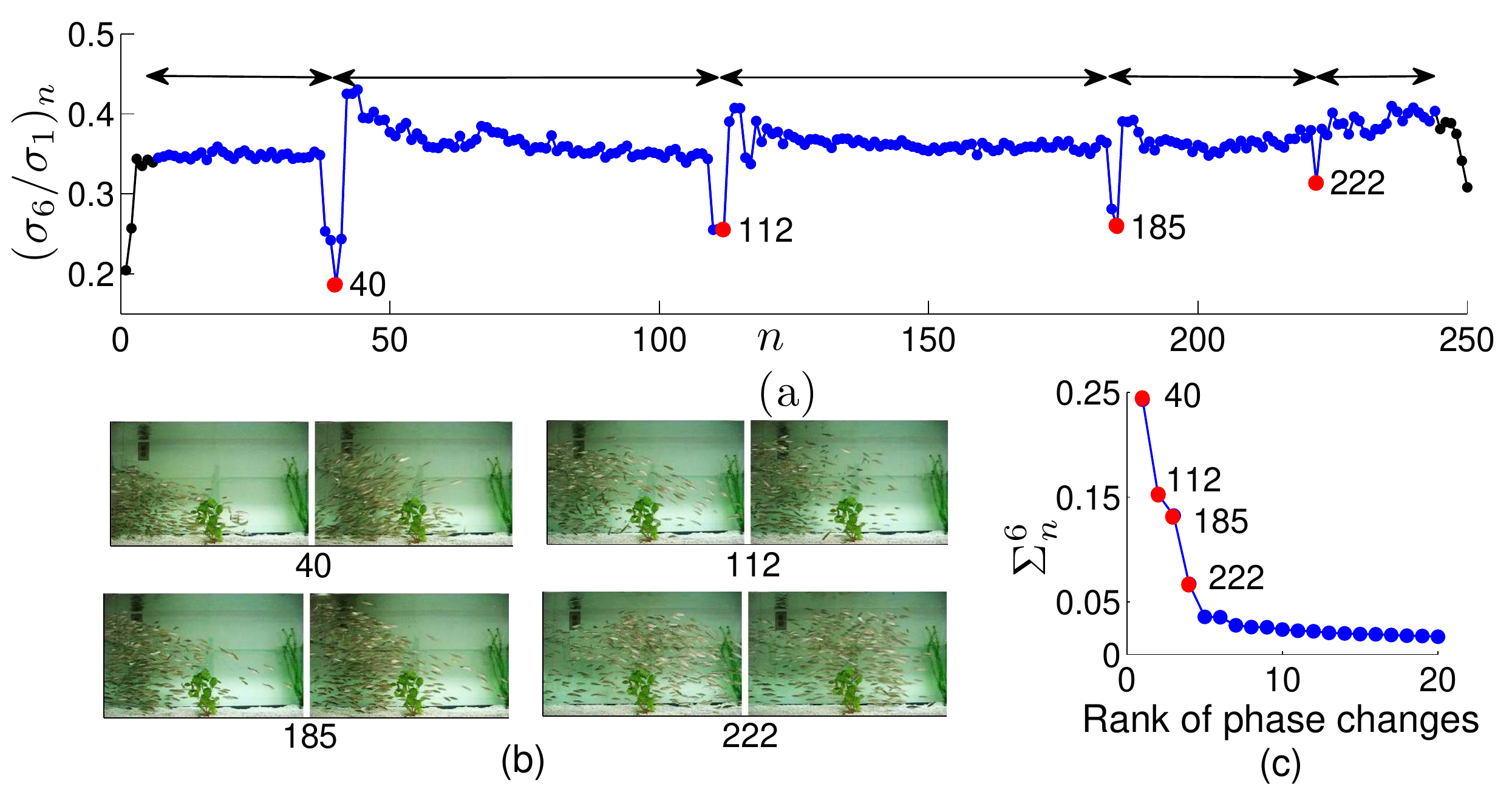}    	
\caption{Detecting phase transitions in a fish school. (a) The distribution of $(\sigma_6/\sigma_1)_{n}$  with left-right arrows showing ranges of frames in sub-manifolds and red dots showing frames at phase transitions. (b) Snapshots of the school before (left) and after (right) each phase transition. (c) The plot of 20 largest phase changes consisting four phase transitions marked in red with their frame numbers.}
\label{fig:example_fish} 
\end{figure}

\section{Conclusions and discussion} \label{sec:conclusion_discussion}
In this study, we proposed a robust method to detect phase transitions in collective behavior using ratio of singular values encountering high curvature of the corresponding underlying manifold. We empirically observe that a phase transition in collective behavior is represented as a locus of singularities on the manifold where the curvature is significantly high. Thus, we first introduced an assertion to approximate the curvature of a curve by means of singular value ratios computed on it and then we extended the assertion to higher dimensions using the shape operator. Finally, we asserted that the same phase transition can also be detected through singular value ratios computed over the local data distribution on the manifold.

We validated the method through four diverse examples; one from a simulation of the Vicsek model \cite{vicsek1995novel} containing three phase transitions, and the other three from natural instances of a crowd of people, a flock of birds, and a school of fish. We ran the phase transition detection algorithm on each data set with a predetermined nearest neighbor parameter ($\alpha$). Algorithm outputs the distribution $(\sigma_{\alpha}/\sigma_1)$ and 20 largest phase changes which are ranked according their magnitudes. Based on these outputs, manifold representing the whole collective behavior is partitioned into phase transitions free sub-manifolds.

In Example ~\ref{sec:example_vicsek}, we simulated a training data set of a particle swam consisting known phase transitions to test the method's performance. Noteworthy shifts of the distribution $(\sigma_4/\sigma_1)_n$ from one frame to the other justify the sensitivity of the method. Therein, our method was capable of detecting the exact phase transitions at frames 50 and 150, however it approximated the phase transition at $100$-th frame as at $99$-th frame with one frame of error. Since this approximation is accurate enough, we partitioned the manifold into four sub-manifolds about frames 50, 99, and 150. Thus, we conclude that the method is adept at perceiving the nature of the collective behavior and spatial distribution of particles to approximate phase transitions. 

The algorithm running on a fish school consisting four phase transitions (Ex. ~\ref{sec:example_fish}) ranked their magnitudes such that it provides a good comparison between them. Embedding dimensionalities of sub-manifolds in each example revealed by isomap affirm that the data is embedded in distinct smooth sub-manifolds which are joined at singularities. Examples ensure the applicability of the method for variety of data sets ranging from simulations to natural, as it requires one input parameter $\alpha$.

As we detect phase transitions based on the curvature which is a local property on the manifold, choosing the best value $\alpha$ is always important for the method to extract correct phase traditions. This is done by either using the prior-knowledge about the data or running the algorithm with several $\alpha$'s to determine the best value which differentiates and highlights phase transitions. We can always rely on a small $\alpha$ value when the data on the manifold is sufficiently dense \cite{tenenbaum2000global}. Since the current techniques of making videos with high frame rates, we can always generate sufficiently large quantity of frames those would then be densely represented as data points on a manifold.  Thus, in this study, we run the phase transition detection algorithm with small $\alpha$ values taken to be all natural numbers less than 10 and then decide the best value by analyzing corresponding phase change plots.

In future, we will fabricate an automated method to determine the best $\alpha$ value which will give the best phase change plot to extract phase transitions. We will also establish a novel approach of detecting phase transitions under gradual changes of the behavior of a multi-agent system which is not addressed in this work.

Here we presented a phase transitions detection method of multi-agent systems usin g the curvature of the  manifold representing the system. The method was validated on several instances ranging from simulation to natural to justify the broad applicability and the accuracy.

\section*{Acknowledgments}
Kelum Gajamannage and Erik M. Bollt are supported by the National Science Foundation under the grant number CMMI- 1129859. Erik M. Bollt was also supported by the Army Research Office under the grant number W911NF-12-1-276 and Office of Naval Research under the grant number N00014-15-2093.

\begin{appendix}
\section{Computing principal sections from the shape operator.} \label{app:shape_operator}
Here, we provide an example of computing principal sections of a two dimensional manifold in $\mathbb{R}^3$ called a saddle surface given in the parametric form
\begin{equation}\label{eqn:2D_saddle}\mathcal{M}^2=(x_1, x_2, x_1^3-3x_1x_2^2) \in \mathbb{R}^3,\end{equation}
shown by Figure ~\ref{fig:principal_curvature_saddle}. 
\begin{figure}[htp]
\centering
\includegraphics[width=.5\textwidth]{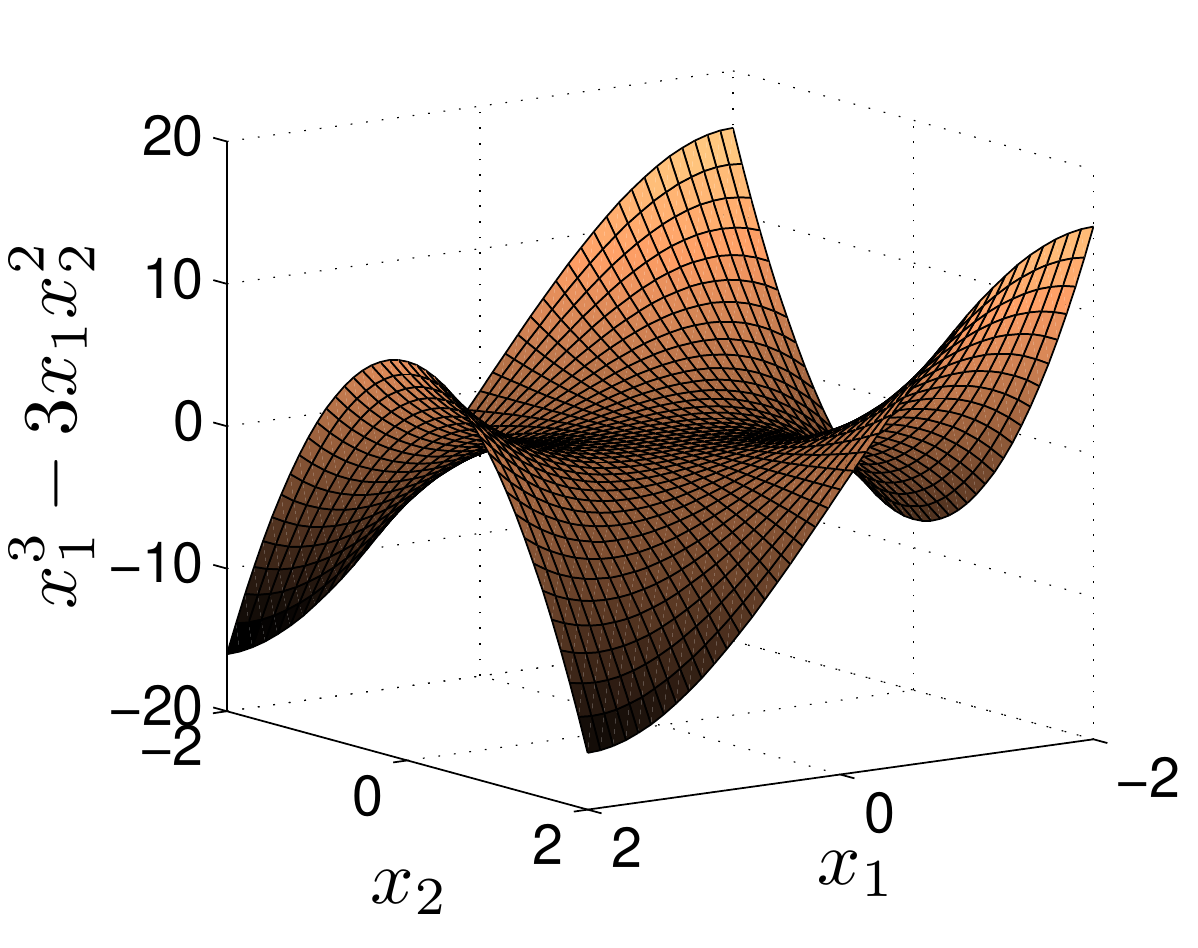}
\caption{Two dimensional saddle surface $\mathcal{M}^2$, described by the Equation (\ref{eqn:2D_saddle}) for $x_1, x_2 \in \mathbb{U} [-2, 2]$.}
\label{fig:principal_curvature_saddle}
\end{figure}

For an arbitrary point $\bs{p}=(x_1, x_2, x_1^3-3x_1x_2^2)$ on the surface, we compute unit tangential vectors as
\begin{equation}\label{eqn:first_tangents}
\bs{v}_{\bs{p}}^{(1)}=\frac{\partial \mathcal{M}^2}{\partial x_1}=\big(1, 0, 3\big[x_1^2-x_2^2\big]\big) \implies \hat{\bs{v}}_{\bs{p}}^{(1)}=\frac{\bs{v}_{\bs{p}}^{(1)}}{\big|\bs{v}_{\bs{p}}^{(1)}\big|}=\frac{\big(1, 0, 3\big[x_1^2-x_2^2\big]\big)}{\sqrt{9\big(x_1^2-x_2^2\big)^2+1}},
\end{equation}
and
\begin{equation}\label{eqn:second_tangents}
\hspace{20pt} \bs{v}_{\bs{p}}^{(2)}=\frac{\partial \mathcal{M}^2}{\partial x_2}=\big(0, 1, -6x_1x_2\big) \implies \hat{\bs{v}}_{\bs{p}}^{(2)}=\frac{\bs{v}_{\bs{p}}^{(2)}}{\big|\bs{v}_{\bs{p}}^{(2)}\big|}=\frac{\big(0, 1, -6x_1x_2\big)}{\sqrt{36x_1^2x_2^2+1}}.
\end{equation} 
Then, while the unit normal at $\bs{p}$, $N_{\bs{p}}$, is
\begin{equation}N_{\bs{p}}=\frac{\bs{v}_{\bs{p}}^{(1)}\times \bs{v}_{\bs{p}}^{(2)}}{\big|\bs{v}_{\bs{p}}^{(1)}\times \bs{v}_{\bs{p}}^{(2)}\big|}=\frac{\big(-3\big[x_1^2-x_2^2\big], 6x_1x_2, 1\big)}{\sqrt{9\big(x_1^2+x_2^2\big)^2+1}},\end{equation}
tangential directions are
\begin{equation}\label{eqn:x_derivative_of_n}\nabla_{x_1} N_{\bs{p}}=\frac{\partial N_{\bs{p}}}{\partial x_1}=\frac{-6\left(x_1\left[18x_2^2\{x_1^2+x_2^2\}+1\right] , x_2\left[9\{x_1^4-x_2^4\}-1\right] , 3x_1\left[x_1^2+x_2^2\right]\right)}{\big[9\big(x_1^2+x_2^2\big)^2+1\big]^{3/2}},
\end{equation}
and
\begin{equation}\label{eqn:y_derivative_of_n}\nabla_{x_2} N_{\bs{p}}=\frac{\partial N_{\bs{p}}}{\partial x_2}=\frac{6\left(x_2\left[18x_1^2\{x_1^2+x_2^2\}+1\right] , x_1\left[9\{x_1^4-x_2^4\}+1\right] , -3x_2\left[x_1^2+x_2^2\right]\right)}{\big[9\left(x_1^2+x_2^2\right)^2+1\big]^{3/2}}.
\end{equation}
By Equations (\ref{eqn:first_tangents}), (\ref{eqn:second_tangents}), (\ref{eqn:x_derivative_of_n}), and (\ref{eqn:y_derivative_of_n}), we compute
\begin{equation}\mathcal{S}_{\bs{p}}^{(1,1)}=-\nabla_{x_1} N_{\bs{p}} \cdot \hat{\bs{v}}_{\bs{p}}^{(1)}=\frac{6x_1}{\sqrt{[9(x_1^2+x_2^2)^2+1][9(x_1^2-x_2^2)^2+1]}},
\end{equation}
\begin{equation}\mathcal{S}_{\bs{p}}^{(1,2)}=-\nabla_{x_1} N_{\bs{p}} \cdot \hat{\bs{v}}_{\bs{p}}^{(2)}=\frac{-6x_2}{\sqrt{[9(x_1^2+x_2^2)^2+1][36x_1^2x_2^2+1]}},
\end{equation}
\begin{equation}
\mathcal{S}_{\bs{p}}^{(2,1)}=-\nabla_{x_2} N_{\bs{p}} \cdot \hat{\bs{v}}_{\bs{p}}^{(1)}=\frac{-6x_2}{\sqrt{[9(x_1^2+x_2^2)^2+1][9(x_1^2-x_2^2)^2+1]}},
\end{equation} 
\begin{equation}\mathcal{S}_{\bs{p}}^{(2,2)}=-\nabla_{x_2} N_{\bs{p}} \cdot \hat{\bs{v}}_{\bs{p}}^{(2)}=\frac{-6x_1}{\sqrt{[9(x_1^2+x_2^2)^2+1][36x_1^2x_2^2+1]}}.
\end{equation} 
Particularly, the shape operator $\mathcal{S}_{\bs{p}}=\begin{pmatrix}\mathcal{S}_{\bs{p}}^{(1,1)} & \mathcal{S}_{\bs{p}}^{(1,2)} \\ \mathcal{S}_{\bs{p}}^{(2,1)} & \mathcal{S}_{\bs{p}}^{(2,2)}\end{pmatrix}$, at $\bs{p}=(1, 0, 1)$ is
\begin{equation}
\mathcal{S}_{(1,0,1)}=\begin{pmatrix}3/5 & 0 \\ 0 & -6/\sqrt{10}\end{pmatrix}.
\end{equation}

Computed eigenpairs, $(3/5, (1, 0))$ and $(-6/\sqrt{10}, (0, 1))$, of $\mathcal{S}_{(1,0,1)}$ describe two pairs of principal curvatures and directions in the tangential space at $\bs{p}=(1, 0, 1)$. These two dimensional principal directions are made to be three dimensions as $(1,0,0)$ and $(0,1,0)$ by introducing the third dimension. The unit normal at $\bs{p}=(1, 0, 1)$ is  $N_{(1,0,1)}=\frac{1}{\sqrt{10}}(-3,0,1)$, thus the principal sections in $\mathbb{R}^3$ at $\bs{p}$ are given by
\begin{equation}
\begin{split}
\Pi^{(1)}_{(1,0,1)}=\left\{\beta (1,0,1) + \gamma (-3,0,1) \vert, \forall \ \gamma, \beta \in \mathbb{R}\right\} \text{and}\\
 \Pi^{(2)}_{(1,0,1)}=\left\{\beta (0,1,1) + \gamma (-3,0,1) \vert, \forall \ \gamma, \beta \in \mathbb{R}\right\}.
\end{split}
\end{equation}

\end{appendix}


\medskip
Received on September 23, 2015. Revised on May 31, 2016.
\medskip


\begin{thebibliography}{99}

\bibitem{birds2011}
\newblock
\newblock Birds flying away, shutterstock.
\newblock Available from: \url{https://www.shutterstock.com/video/clip-3003274-stock-footage-birds-flying-away.html?src=search/Yg-XYej1Po2F0VO3yykclw:1:19/gg}.

\bibitem{detection}
\newblock
\newblock Data set of detection of unusual crowd activity available at robotics and vision laboratory, Department of Computer Science and Engineering, University of Minnesota.
\newblock Available from: \url{http://mha.cs.umn.edu/proj\_events.shtml}.

\bibitem{pet2009}
\newblock
\newblock Data set of pet2009 at Computational Vision Group, University of Reading,
\newblock 2009. Available from: \url{http://ftp.pets.reading.ac.uk/pub/}.

\bibitem{abaid2012topological}
\newblock N. Abaid, E. Bollt, and M. Porfiri,
\newblock Topological analysis of complexity in multiagent systems,
\newblock {\em Physical Review E}, \textbf{85}(2012), 041907.

\bibitem{alfred1998modern}
\newblock G. Alfred,
\newblock {\em Modern Differential Geometry of Curves and Surfaces with Mathematica},
\newblock CRC press, 1998.

\bibitem{almeida2013change}
\newblock I. R. de Almeida and C. R. Jung,
\newblock Change detection in human crowds,
\newblock in {\em Graphics, Patterns and Images (SIBGRAPI), 2013 26th  SIBGRAPI-Conference on}, IEEE, (2013), 63--69.

\bibitem{andrade2006hidden}
\newblock E. L.  Andrade, S. Blunsden, and R. B. Fisher,
\newblock Hidden markov models for optical flow analysis in crowds,
\newblock in {\em Pattern Recognition, 2006. ICPR 2006. 18th International  Conference on}, IEEE, \textbf{1} (2006), 460--463.

\bibitem{ballerini2008empirical}
\newblock M. Ballerini, N. Cabibbo, R. Candelier, A. Cavagna, E. Cisbani, I. Giardina, A. Orlandi, G. Parisi, A. Procaccini, M. Viale, et al,
\newblock Empirical investigation of starling flocks: a benchmark study in collective animal behavior,
\newblock {\em Animal Behaviour}, \textbf{76} (2008), 201--215.

\bibitem{becco2006experimental}
\newblock C. Becco, N. Vandewalle, J. Delcourt, and P. Poncin,
\newblock Experimental evidences of a structural and dynamical transition in   fish school,
\newblock {\em Physica A: Statistical Mechanics and its Applications}, \textbf{367} (2006), 487--493.

\bibitem{beekman2001phase}
\newblock M. Beekman, D. J. T. Sumpter, and F. L. W. Ratnieks,
\newblock Phase transition between disordered and ordered foraging in pharaoh's ants,
\newblock {\em Proceedings of the National Academy of Sciences}, \textbf{98} (2001), 9703--9706.

\bibitem{bovik2010handbook}
\newblock A. C. Bovik,
\newblock {\em Handbook of Image and Video Processing},
\newblock Academic press, 2010.

\bibitem{bracewell2003fourier}
\newblock R. Bracewell,
\newblock {\em Fourier Analysis and Imaging},
\newblock Springer Science \& Business Media, 2010.

\bibitem{couzin2009collective}
\newblock I. D. Couzin,
\newblock Collective cognition in animal groups,
\newblock {\em Trends in cognitive sciences}, \textbf{13} (2009), 36--43.

\bibitem{couzin2005effective}
\newblock I. D. Couzin, J. Krause, N. R. Franks, and S. A. Levin,
\newblock Effective leadership and decision-making in animal groups on the move,
\newblock {\em Nature}, \textbf{433} (2005), 513--516.

\bibitem{deutsch1999principles}
\newblock A. Deutsch,
\newblock Principles of biological pattern formation: swarming and aggregation viewed as self organization phenomena,
\newblock {\em Journal of Biosciences}, \textbf{24} (1999), 115--120.

\bibitem{friedman1977algorithm}
\newblock J. H. Friedman, J. L. Bentley, and R. A. Finkel,
\newblock An algorithm for finding best matches in logarithmic expected time,
\newblock {\em ACM Transactions on Mathematical Software}, \textbf{3} (1977), 209--226.

\bibitem{gajamannage2015dimensionality}
\newblock K. Gajamannage, S. Butailb, M. Porfirib, and E. M. Bollt,
\newblock Model reduction of collective motion by principal manifolds,
\newblock {\em Physica D: Nonlinear Phenomena}, \textbf{291} (2015), 62-73. 

\bibitem{gajamannageidentifying}
\newblock K. Gajamannage, S. Butailb, M. Porfirib, and E. M. Bollt,
\newblock Identifying manifolds underlying group motion in Vicsek agents,
\newblock {\em The European Physical Journal Special Topics}, \textbf{224} (2015), 3245--3256.

\bibitem{gerbrands1981}
\newblock J. J. Gerbrands,
\newblock On the relationships between SVD, KLT and PCA,
\newblock {\em Pattern recognition}, \textbf{14} (1981), 375-381.

\bibitem{gerlai2010high}
\newblock R. Gerlai,
\newblock High-throughput behavioral screens: the first step towards finding genes involved in vertebrate brain function using zebra fish,
\newblock {\em Molecules}, \textbf{15} (2010), 2609--2622.

\bibitem{golub1970singular}
\newblock G. H. Golub and C. Reinsch,
\newblock Singular value decomposition and least squares solutions,
\newblock {\em Numerische Mathematik}, \textbf{14} (1970), 403--420.

\bibitem{helbing1997modelling}
\newblock D. Helbing, J. Keltsch, and P. Molnar,
\newblock Modelling the evolution of human trail systems,
\newblock {\em Nature}, \textbf{388} (1997), 47--50.

\bibitem{lee1997riemannian}
\newblock J. M. Lee,
\newblock {\em Riemannian Manifolds: an Introduction to Curvature}, volume 176,
\newblock Springer, 1997.

\bibitem{lee2012introduction}
\newblock J. M. Lee,
\newblock {\em Introduction to Smooth Manifolds},
\newblock Springer, 2012.

\bibitem{mehran2009abnormal}
\newblock R. Mehran, A. Oyama, and M. Shah,
\newblock Abnormal crowd behavior detection using social force model,
\newblock in {\em Computer Vision and Pattern Recognition, 2009. CVPR 2009. IEEE Conference on}, IEEE, (2009), 935--942.

\bibitem{millonas1992swarms}
\newblock M. M. Millonas,
\newblock {\em Swarms, phase transitions, and collective intelligence},
\newblock Technical report, Los Alamos National Lab., New Mexico, USA, 1992.

\bibitem{musse1997model}
\newblock S. R. Musse and D. Thalmann,
\newblock A model of human crowd behavior: group inter-relationship and collision detection analysis,
\newblock in {\em Computer Animation and Simulation}, Springer, (1997), 39--51.

\bibitem{nagy2010hierarchical}
\newblock M. Nagy, Z. {\'A}kos, D. Biro, and T. Vicsek,
\newblock Hierarchical group dynamics in pigeon flocks,
\newblock {\em Nature}, \textbf{464} (2010), 890--893.

\bibitem{o1966elementary}
\newblock B. O'neill,
\newblock {\em Elementary Differential Geometry}, 
\newblock Academic press, New York, 1966.

\bibitem{papenbrock2002invariant}
\newblock T. Papenbrock and T. H. Seligman,
\newblock Invariant manifolds and collective motion in many-body systems,
\newblock reprint, \arXiv{nlin/0206035}.

\bibitem{partridge1982structure}
\newblock B. L. Partridge,
\newblock The structure and function of fish schools,
\newblock {\em Scientific American}, \textbf{246} (1982), 114--123.

\bibitem{rappel1999self}
\newblock W. Rappel, A. Nicol, A. Sarkissian, H. Levine, and W. F. Loomis,
\newblock Self-organized vortex state in two-dimensional dictyostelium   dynamics,
\newblock {\em Physical Review Letters}, \textbf{83} (1999), 1247.

\bibitem{rauch1995pattern}
\newblock E. M. Rauch, M. M. Millonas, and D. R. Chialvo,
\newblock Pattern formation and functionality in swarm models,
\newblock {\em Physics Letters A}, \textbf{207} (1995), 185--193.

\bibitem{rovenskii2011topics}
\newblock V. Y. Rovenskii,
\newblock {\em Topics in Extrinsic Geometry of Codimension-one Foliations}, 
\newblock Springer, 2011.

\bibitem{roweis2000nonlinear}
\newblock S. T. Roweis and L. K. Saul,
\newblock Nonlinear dimensionality reduction by locally linear embedding,
\newblock {\em Science}, \textbf{290} (2000), 2323--2326.

\bibitem{sole1996phase}
\newblock R. V. Sol{\'e}, S. C. Manrubia, B. Luque, J. Delgado, and J. Bascompte,
\newblock Phase transitions and complex systems: simple, nonlinear models capture complex systems at the edge of chaos,
\newblock {\em Complexity}, \textbf{1} (1996), 13--26.

\bibitem{somasundaram2005differential}
\newblock D. Somasundaram,
\newblock {\em Differential Geometry: A First Course},
\newblock Alpha Science Int'l Ltd., 2005.

\bibitem{sumpter2008information}
\newblock D. Sumpter, J. Buhl, D. Biro, and I. Couzin,
\newblock Information transfer in moving animal groups,
\newblock {\em Theory in Biosciences}, \textbf{127} (2008), 177--186.

\bibitem{tenenbaum2000global}
\newblock J. B. Tenenbaum, V. De Silva, and J. C. Langford,
\newblock A global geometric framework for nonlinear dimensionality reduction,
\newblock {\em Science}, \textbf{290} (2000), 2319--2323.

\bibitem{toffin2009shape}
\newblock E. Toffin, D. D. Paolo, A. Campo, C. Detrain, and J. Deneubourg,
\newblock Shape transition during nest digging in ants,
\newblock {\em Proceedings of the National Academy of Sciences}, \textbf{106} (2009):18616--18620.

\bibitem{topaz2004swarming}
C. M. Topaz and A. L. Bertozzi,
\newblock Swarming patterns in a two-dimensional kinematic model for biological groups,
\newblock {\em SIAM Journal on Applied Mathematics}, \textbf{65} (2004), 152--174.

\bibitem{vicsek1995novel}
\newblock T. Vicsek, A. Czir{\'o}k, E. Ben-Jacob, I. Cohen, and O. Shochet,
\newblock Novel type of phase transition in a system of self-driven particles,
\newblock {\em Physical Review Letters}, \textbf{75} (1995), 1226.

\bibitem{witten1996phase}
\newblock E. Witten,
\newblock Phase transitions in m-theory and f-theory,
\newblock {\em Nuclear Physics B}, \textbf{471} (1996), 195--216.

\bibitem{yianilos1993data}
\newblock P. N. Yianilos,
\newblock Data structures and algorithms for nearest neighbor search in general metric spaces,
\newblock in {\em Proceedings of the Fourth Annual ACM-SIAM Symposium on Discrete Algorithms}, Society for Industrial and Applied Mathematics, (1993), 311--321.

\bibitem{zhao2004tracking}
\newblock T. Zhao and R. Nevatia,
\newblock Tracking multiple humans in crowded environment,
\newblock in {\em Computer Vision and Pattern Recognition, 2004. CVPR 2004. Proceedings of the 2004 IEEE Computer Society Conference on}, IEEE, (2004), II--406.

\end{thebibliography}
\end{document}